\documentclass[a4paper,11pt]{article}

\usepackage[top=3.0cm,bottom=3.0cm,left=2.5cm,right=2.5cm]{geometry}
\usepackage{tikz}
\usepackage{graphics,epsfig,psfrag}
\usepackage{graphicx}
\usepackage{amsfonts}
\usepackage{amssymb}
\usepackage{amsmath}
\usepackage{amsthm}
\usepackage[english]{babel}
\usepackage{ctex}
\usepackage{hyperref}
\usepackage{cases}
\usepackage{authblk}
\usepackage{mathtools}
\usepackage{pifont}

\newtheorem{lemma}{Lemma}[section]
\newtheorem{theorem}[lemma]{Theorem}

\newtheorem{claim}[]{\noindent Claim}[section]

\newtheorem{conjecture}[lemma]{Conjecture}

\begin{document}

\title{Odd complete bipartite minors in graphs with\\
 independence number two}

        \author[1]{Rong Chen\footnote{Email: rongchen@fzu.edu.cn (R. Chen).}}
		\author[1]{Zijian Deng\footnote{Email: zj1329205716@163.com.(corresponding author).}}

\affil[1]{Center for Discrete Mathematics, Fuzhou University,
Fujian, 350003, China}

	\date{}
	\maketitle
	\begin{abstract}
Recently, Chen and Deng have proved that every graph $G$ with independence number two contains $K_{\ell,\chi(G)- \ell}$ as a minor for each integer $\ell$ with $1\leq\ell < \chi(G)$. In this paper, we extend this result to odd minor version. That is, we prove that  each graph $G$ with independence number two contains $K_{\ell,\chi(G)- \ell}$ as an odd minor for each integer $\ell$ with $1\leq\ell < \chi(G)$.


	\end{abstract}
	
	\textbf{Mathematics Subject Classification}: 05C15; 05C83
	
	\textbf{Keywords}: Hadwiger's conjecture; odd minors; independence number


\section{\bf Introduction}
All graphs considered in this paper are finite and simple. 
For a graph $G$, we use $\chi(G)$ and $\omega(G)$ to denote the chromatic number and clique number of $G$, respectively.
An \emph{independent set} 
is a subset of $V(G)$ that are pairwise nonadjacent. Let $\alpha(G)$ be the \emph{independence number} of $G$, that is the maximum size of independent sets. Let $K_{n}$ be a clique on $n$ vertices, and $P_{n}$ a path on $n$ vertices. 

Let \(H\) be a graph with vertex set \(V(H) = \{v_1, v_2, \dots, v_h\}\). An \emph{\(H\)-model} in a graph \(G\) is a collection of vertex-disjoint subgraphs \((T_v)_{v \in V(H)}\) in \(G\), where each \(T_v\) is a tree (called a \emph{branch set} of $H$), and for every edge \(uv \in E(H)\), there exists at least one edge in \(G\) connecting a vertex of \(T_u\) to a vertex of \(T_v\).
We call an $H$-model \emph{odd} if there exists a vertex coloring of \( H \) using two colors (e.g., \(\{1, 2\}\)) such that:
\begin{itemize}
    \item Every edge within each tree \( T_s \) is \emph{bichromatic} (i.e., its endpoints have different colors);
    \item Every edge connecting two distinct branch sets \(T_u\) and \(T_v\) (corresponding to an edge \(uv \in E(H)\)) is \emph{monochromatic} (i.e., its endpoints have the same color).
\end{itemize}

\noindent We call an odd $H$-model \emph{special} if all branch sets that are isolated vertices share the same color. Evidently, when $H$ is a clique, each odd $H$-model is special.

We say a graph $G$ contains $H$ as an (odd) minor 
if it includes a subgraph that is an (odd) $H$-model. For simplicity, we use $G\succeq_m H$ ($G\succeq_{om} H$) to represent the fact that $H$ is an (odd) minor of $G$.

In 1943, Hadwiger \cite{HU} proposed the following famous conjecture.

\begin{conjecture}{\rm(Hadwiger's conjecture)}\label{HC}
 Every graph $G$ contains $K_{\chi(G)}$ as a minor.
\end{conjecture}

Hadwiger's conjecture has been proven to be true only for graphs with chromatic number at most 6, but is likely difficult to prove for larger values, given that proofs for the cases of $\chi(G)=5$ and $\chi(G)=6$ already depend on the Four Color Theorem.
For more background on Hadwiger's conjecture, we refer to Seymour's 2016-survey \cite{P} and a recent update on the developments by Norin \cite{N}.

To address the difficulty of resolving Hadwiger's conjecture for general chromatic numbers, Woodall \cite{W} in 2001 (and independently Seymour \cite{AN}) proposed a weaker variant that relaxes the requirement of a complete minor.

\begin{conjecture}{\rm(\cite{W,AN})}\label{conj1w}
Every graph $G$ contains $K_{\ell, \chi(G)-\ell}$ as a minor for any positive integer $\ell$ with $\ell< \chi(G)$.
\end{conjecture}

While this weaker conjecture seeks to simplify the original problem, another line of research has focused on strengthening Hadwiger's conjecture by imposing stricter conditions on the minor structure.
Gerards and Seymour (see \cite{TB}, Section 6.5) proposed the following strengthening of Hadwiger's conjecture, called Odd Hadwiger's conjecture.

\begin{conjecture}{\rm(Odd Hadwiger's conjecture)}\label{OHC}
 Every graph $G$ contains $K_{\chi(G)}$ as an odd minor.
\end{conjecture}

Given that Conjecture \ref{OHC} strengthens Hadwiger's long standing and notoriously difficult conjecture, it is unsurprising that it has resisted all resolution attempts to date. Catlin \cite{PC} established the cases for \( \chi(G) \leq 4 \), and Guenin announced a proof for \( \chi(G) = 5 \) (strengthening the Four Color Theorem) in 2004 \cite{P}.
Similar to Hadwiger's conjecture, the exact formulation of Conjecture \ref{OHC} remains unresolved for
\( \chi(G) \geq 6 \). Nevertheless, substantial progress has been achieved in asymptotic analyses
(\cite{JB,SN,LP2,S}), culminating in the ground-breaking work of \cite{S},
which proves an $O(t \log \log t)$ upper bound on the chromatic number of graphs excluding odd \( K_t \)-minors.
For more results around the Odd Hadwiger's
conjecture, we refer the interested reader to Section 7 of the survey \cite{P}.

Hadwiger's conjecture remains unsolved for a specific class of graphs, namely those with an independence number of at most two.

\begin{conjecture}\label{conj2}
Every graph $G$ with $\alpha(G)\leq2$ contains $K_{\chi(G)}$ as a minor.
\end{conjecture}

The significance of this class of graphs with independence number two was first pointed out by Duchet and Meyniel \cite{PH} and Mader independently, and later highlighted in Seymour's survey paper \cite{P}.
According to Seymour in \cite{P}, 
Conjecture \ref{conj2} could be key to the Hadwiger's conjecture. He suggested that if the conjecture is valid for these graphs, it might be valid more broadly. As an evidence for Conjecture \ref{conj1w}, Chen and Deng \cite{CD} have shown the following result.

\begin{theorem} \label{biminor}{\rm(\cite{CD})}
Let $G$ be a graph with $\alpha(G)\leq2$. For any positive integer $\ell$ with $\ell < \chi(G)$, we have $G \succeq_{m} K_{\ell,\chi(G)- \ell}$.
\end{theorem}







In this paper, we extend Theorem \ref{biminor} to the odd minor version.


\begin{theorem} \label{obiminor2}
Let $G$ be a graph with $\alpha(G)\leq2$. For any positive integer $\ell$ with $\ell < \chi(G)$, we have $G \succeq_{om} K_{\ell,\chi(G)- \ell}$.
\end{theorem}

Actually, Chen and Deng in \cite{CD} showed that each graph $G$ with $\alpha(G)\leq2$ contains a minor of $K^{\ell}_{\ell,\chi(G)- \ell}$ for each positive integer $\ell$ with $2\ell \leq \chi(G)$, where  $K_{m,n}^m$ is obtained from the disjoint union of a $K_m$ and an independent set on $n$ vertices by adding all of the possible edges between them. Considered this fact, we conjecture that a similar result also holds for odd minor.

\begin{conjecture}\label{conj3}
Let $G$ be a graph with $\alpha(G)\leq2$. For any positive integer $\ell$ with $2\ell \leq \chi(G)$, we have $G \succeq_{om} K^{\ell}_{\ell,\chi(G)- \ell}$.
\end{conjecture}

As an evidence of Conjecture \ref{conj3}, in Section 2, we prove that each $n$-vertex graph with independence number at most two contains an odd minor of $ K^{\ell}_{\ell,\lceil\frac{n}{2}\rceil- \ell}$ for any positive integer $\ell$ with $2\ell \leq \lceil\frac{n}{2}\rceil$ (see Theorem \ref{main-lemma}).


\section{Proof of Theorem \ref{obiminor2}}
For a graph $G$, let  $\kappa(G)$ be the connectivity of $G$. For any set $X\subseteq V(G)$, let $G[X]$ denote the subgraph of $G$ induced by $X$, and $G \backslash X$ denote the subgraph $G[V \backslash X]$. For any set $\mathcal{A}$ of vertex disjoint subgraphs of $G$, let $V(\mathcal{A})$ denote the union of the vertex sets of subgraphs in $\mathcal{A}$.
Given disjoint vertex sets $A$ and $B$, we say that $A$ is \emph{complete} to $B$ if  each vertex in $A$ is adjacent to all vertices in $B$.
We denote by $N_{G}(A)$ the vertex set in $V(G)-A$ that has adjacent vertices in $A$.
Set $N_{G}[A]:=N_{G}(A)\cup A$.
When there is no danger of confusion, all subscripts will be omitted.
For simplicity, when $A = \{a\}$, the sets $N(\{a\})$ and $N[\{a\}]$ are denoted by $N(a)$  and $N[a]$, respectively.

In \cite{J}, Blasiak shows that for any $n$-vertex graph $G$ with $\alpha(G) \leq 2$, the absence of a minor of $K_{\lceil\frac{n}{2}\rceil}$ implies $\kappa(G) \geq \lceil\frac{n}{2}\rceil$. 
Following the ideas in Blasiak's proof, we can prove that similar result holds for odd minors.

\begin{lemma} \label{thm0}
Let $G$ be an $n$-vertex graph with $\alpha(G)\leq2$. If $G$ is not $\lceil\frac{n}{2}\rceil$-connected, then $G\succeq_{om} K_{\lceil\frac{n}{2}\rceil}$. %
\end{lemma}
\begin{proof}
Suppose $G$ is not $\lceil\frac{n}{2}\rceil$-connected.
Choose a vertex cutset $X$ of $G$ as small as possible. Since $\alpha(G)\leq2$ and $X$ is a cutset of $G$, the graph $G\backslash X$ has exactly two components, say $L$ and $R$, and both of them are cliques. Since $X$ is a minimal cutset, every vertex in $X$ is either complete to $L$ or complete to $R$.
Let $X_L, X_R$ partition $X$ such that $X_L$ ($X_R$, resp.) is complete to $L$ ($R$, resp.). Without loss of generality we may assume $|V(L)| + |X_L| \geq |V(R)| + |X_R|$. 
Then $|V(L)|+|X_L|\geq \lceil\frac{n}{2}\rceil$. 

We claim that for any subset $A \subseteq X_L$ of size at most $|V(R)|$, the set $A$ is matchable into $V(R)$. Suppose not. By Hall's matching condition, there exists $S \subseteq A$ such that $|S| > |N(S) \cap V(R)|$. Then the set $(X - S) \cup (N(S) \cap V(R))$ is a vertex cutset of $G$. 
This contradicts the minimality of $X$ since $|(X - S) \cup (N(S) \cap V(R))| < |X|$.

Let $M$ be a matching from $X_L$ to $V(R)$ of size $\min(|X_L|, |V(R)|)$. Since the ends of every pair of edges of $M$ in $R$ are adjacent, all vertices of $L$, together with all edges of $M$, form an odd complete minor model, because 
we can assign color 1 to $V(L)\cup (V(M) \cap X_L)$, while assigning color 2 to $V(M) \cap V(R)$.
Since $|X| < \lceil\frac{n}{2}\rceil$, we have $|L|+|V(R)|\geq \lceil\frac{n}{2}\rceil$.
Moreover, since $|V(L)|+|X_L|\geq \lceil\frac{n}{2}\rceil$, the size of the odd complete minor is
$$|V(L)| + \min(|X_L|, |V(R)|) = \min\left(|V(L)| + |X_L|, |V(L)| + |V(R)|\right) \geq \lceil\frac{n}{2}\rceil.$$ This proves Lemma \ref{thm0}.
\end{proof}

\begin{theorem} \label{thmc}{\rm(\cite{MP})}
Let $G$ be an $n$-vertex graph with $\alpha(G)\leq2$ and \( G \) be $\lceil\frac{n}{2}\rceil$-connected. If the largest clique \( K \) in \( G \) satisfies
\[
\frac{3}{2} \left\lceil \frac{n}{2} \right\rceil - \frac{n}{2} \leq |V(K)| \leq \lceil\frac{n}{2}\rceil,
\]
then there exist \( \lceil\frac{n}{2}\rceil - |V(K)| \) pairwise disjoint induced $3$-vertex-paths in \( G\backslash V(K) \).
\end{theorem}


\begin{lemma} \label{thm2}
Let $G$ be an $n$-vertex graph with $\alpha(G)\leq2$. If $n$ is odd and $\omega(G)\geq \frac{n+3}{4}$, then   
$G\succeq_{om} K_{\lceil\frac{n}{2}\rceil}$.
\end{lemma}

\begin{proof}
Assume not. Then $\omega(G) < \lceil\frac{n}{2}\rceil$ and it follows from Lemma \ref{thm0} that $G$ is $\lceil\frac{n}{2}\rceil$-connected. Choose a clique $K$ having $\omega(G)$ vertices. Since $n$ is odd, $\frac{3}{2} \left\lceil \frac{n}{2} \right\rceil - \frac{n}{2}=\frac{n+3}{4} \leq \omega(G) < \lceil\frac{n}{2}\rceil$.
By Theorem \ref{thmc}, there exist \( \lceil\frac{n}{2}\rceil - \omega(G) \) pairwise disjoint induced 3-vertex-paths in \( G\backslash V(K) \). Let $\mathcal{P}$ be the collection of these paths.
To verify  that $G$ contains a special odd model of $K_{\lceil\frac{n}{2}\rceil}$,
we choose each member of \(\mathcal{P}\) and each vertex in $V(K)$ as the branch sets of $K_{\lceil\frac{n}{2}\rceil}$, and assign color 1 to $V(K)$ and the end vertices of each path in $\mathcal{P}$, and color 2 to the internal vertices of each path in $\mathcal{P}$. Hence, $G$ contains a special odd model of $K_{\lceil\frac{n}{2}\rceil}$, which is a contradiction.
\end{proof}








\begin{lemma} \label{lem1}{\rm(\cite{CD})}
Let $G$ be an $n$-vertex  graph with $\alpha(G)\leq2$, with $\omega(G)<\lceil\frac{n}{2}\rceil$, and with $\kappa(G)\geq \lceil\frac{n-1}{4}\rceil$. For any integer $\ell$ with $\ell \leq \lceil\frac{n}{4}\rceil$, if $n\geq 4\ell+1$ is odd, then $G$ has $\ell$ pairwise disjoint induced $3$-vertex-paths $P_3$.
\end{lemma}

\begin{theorem} \label{main-lemma}
Let $G$ be an $n$-vertex graph with $\alpha(G)=2$. For any positive integer $\ell$ with $2\ell \leq \lceil\frac{n}{2}\rceil$, $G$ contains a special odd model of $ K^{\ell}_{\ell,\lceil\frac{n}{2}\rceil- \ell}$. 
\end{theorem}
\begin{proof}
Assume that Theorem \ref{main-lemma} is not true. Let $G$ be a counterexample to Theorem \ref{main-lemma} with $|V(G)|$ as small as possible.

\begin{claim}\label{claim0}
$n\geq4\ell-1$ is odd.
\end{claim}
\begin{proof}
Assume that $n$ is even. Since Theorem \ref{main-lemma} holds for $G \backslash v$ for any $v\in V(G)$, Theorem \ref{main-lemma} holds for $G$ as $\lceil\frac{n}{2}\rceil=\lceil\frac{n-1}{2}\rceil$, a contradiction. Therefore, $n$ is odd, so $n\geq4\ell-1$ as $2\ell \leq \lceil\frac{n}{2}\rceil$. 
\end{proof}

Since all odd models of a clique are special, by Lemmas \ref{thm0} and \ref{thm2}, and Claim \ref{claim0}, we have $\kappa(G)\geq \lceil\frac{n}{2}\rceil$ and $\omega(G)< \frac{n+3}{4}$.

\begin{claim}\label{claim1}
$n=4\ell-1$.
\end{claim}
\begin{proof}
Suppose not. Since $n$ is odd by Claim \ref{claim0}, we have $n\geq4\ell+1$. Since $\kappa(G)\geq \lceil\frac{n}{2}\rceil \geq \lceil\frac{n-1}{4}\rceil $
and $\omega(G)< \frac{n+3}{4} \leq \lceil\frac{n}{2}\rceil$,
it follows from Lemma \ref{lem1} that 
$G$ contains $\ell$ pairwise disjoint induced subgraphs, each isomorphic to $P_3$, denoted the collection of such subgraphs  by $\mathcal{P}$. 
Moreover, since $n\geq 4\ell+1$ is odd by Claim \ref{claim0}, we have $$|V(G)-V(\mathcal{P})|=n-3\ell=\lceil\frac{n}{2}\rceil -\ell+\lfloor\frac{n}{2}\rfloor-2\ell\geq \lceil\frac{n}{2}\rceil -\ell. \eqno{(2.1)}$$

We claim that $G$ contains a special odd model of $K^{\ell}_{\ell,\lceil\frac{n}{2}\rceil- \ell}$. Let $B$ be a subset of $V(G)-V(\mathcal{P})$ with size $\lceil\frac{n}{2}\rceil-\ell$. By (2.1), such $B$ exists. Let \(\mathcal{P}\) be the set of branch sets of the part of size \(\ell\), and the vertices in \(B\) be the branch sets for the part of size $\lceil\frac{n}{2}\rceil-\ell$. Assign color 2 to all endpoints of each path $P_3$ in \(\mathcal{P}\) and all vertices in \(B\), and color 1 to the internal vertex of each path $P_3$ in \(\mathcal{P}\). Then \(\mathcal{P}\cup\{v:\ v\in B\}\) induces an odd model of $ K^{\ell}_{\ell,\lceil\frac{n}{2}\rceil- \ell}$ of $G$ with all branch sets of the part of size $\lceil\frac{n}{2}\rceil-\ell$ having the same color  as $\alpha(G)=2$,
which is a contradiction.
\end{proof}

For any vertex $v\in V(G)$, set $M(v):=V(G)-N[v]$. Since $\alpha(G)=2$, the set $M(v)$ induces a clique in $G$. Moreover, since $\omega(G)< \frac{n+3}{4}$ and $n=4\ell-1$ by Claim \ref{claim1}, we have $\omega(G)\leq \ell$. Hence, $|M(v)|\leq \ell$. 

\begin{claim}\label{claim2}
For any pair of vertices $\{u,v\}$ in $V(G)$, the graph $G \backslash \{u,v\}$ contains $\ell-1$ pairwise disjoint induced subgraphs, each isomorphic to $P_3$.
\end{claim}
\begin{proof}
Since $\kappa(G)\geq \lceil\frac{n}{2}\rceil$, we have $\kappa(G \backslash \{u,v\})\geq \lceil\frac{n}{2}\rceil-2= \lceil\frac{n-2}{2}\rceil-1 \geq \lceil\frac{|V(G \backslash \{u,v\})|-1}{4}\rceil $.
Since Theorem \ref{main-lemma} trivially holds when $n\leq 3$, we have $n\geq5$ as $n$ is odd by Claim \ref{claim0}. Then  $\omega(G \backslash \{u,v\})\leq\omega(G) < \frac{n+3}{4} \leq \lceil\frac{|V(G \backslash \{u,v\})|}{2}\rceil$. Hence,
by Lemma \ref{lem1}, $G \backslash \{u,v\}$ has $\ell-1$ pairwise disjoint induced subgraphs, each isomorphic to $P_3$.
\end{proof}

By Claim \ref{claim2}, for any vertex $v\in V(G)$, the graph $G \backslash v$ contains $\ell-1$ pairwise disjoint induced subgraphs, each isomorphic to $P_3$, denoted the collection of such subgraphs  by $\mathcal{P}_v$.
Set $B:=V(G)-(V(\mathcal{P}_v)\cup \{v\})$.
Then $|B|=\ell+1$ by Claim \ref{claim1}.
Without loss of generality we may assume that $\mathcal{P}_v$ is chosen with $|N_G[v]-V(\mathcal{P}_v)|$ as large as possible. Set $X: = B \cap M(v)$ and $Y:= B \cap N(v)$. That is, $(X,Y)$ is a partition of $B$.

\begin{claim}\label{claim4}
$|X|\leq 1$ and $|Y|\geq \ell$.
\end{claim}

\begin{proof}
Assume that $|X|\geq 2$. Since $|M(v)|\leq \ell$ and $X \subseteq M(v)$, we have $|M(v)\cap \mathcal{P}_v|\leq \ell-2$. Thus, there must exist an induced 3-vertex-path $P \in \mathcal{P}_v$ such that $v$ is complete to $V(P)$. Set $P:=a_1$-$a_2$-$a_3$. 
Choose $x\in X$. Since $\alpha(G)=2$, by symmetry we may assume that $a_1x\in E(G)$.
When $a_3x\in E(G)$, replacing $P$ with $a_3$-$x$-$a_1$ get $\ell-1$ pairwise disjoint induced 3-vertex-paths $\mathcal{P}$ in $G \backslash v$.
Since $v$ is complete to $V(P)$, we have $|N_G[v]-V(\mathcal{P})|>|N_G[v]-V(\mathcal{P}_v)|$, which is a contradiction to the choice of $\mathcal{P}_v$. So $a_3x\notin E(G)$. Similarly, when $a_2x\notin E(G)$,  we can replace $P$ with $a_2$-$a_1$-$x$ to get a contradiction to the choice of $\mathcal{P}_v$. So $a_2x\in E(G)$. Replacing $P$ with $a_3$-$a_2$-$x$ get a contradiction to the choice of $\mathcal{P}_v$. Thus $|X|\leq 1$. Since $|B|=\ell+1$ and $(X,Y)$ is a partition of $B$, we have $|Y|\geq \ell$.
\end{proof}

We claim that $G$ contains a special odd $K^{\ell}_{\ell,\lceil\frac{n}{2}\rceil- \ell}$-model, implying that Theorem \ref{main-lemma} holds. 
Arbitrary choose $Y' \subseteq Y$ with $|Y'|=\ell=\lceil\frac{n}{2}\rceil-\ell$. Since $|Y|\geq \ell$ by Claim \ref{claim4}, such $Y'$ exists.
Let \(\mathcal{P}_v\cup \{v\}\) be the set of branch sets  of the part of size \(\ell\), and all vertices in  $Y'$ be the branch sets  for the part of size $\lceil\frac{n}{2}\rceil-\ell$. Assign color 2 to all endpoints of each path $P_3$ in \(\mathcal{P}_v\) and all vertices in $Y' \cup \{v\}$, and color 1 to the internal vertex of each path $P_3$ in \(\mathcal{P}_v\).
Then $V(\mathcal{P}_v)\cup Y' \cup \{v\}$ induces an odd model of $ K^{\ell}_{\ell,\lceil\frac{n}{2}\rceil- \ell}$ with all branch sets that are isolated vertices having the same color as $\alpha(G)=2$. 
\end{proof}

A graph $G$ is \emph{$k$-vertex-critical} if $\chi(G) = k$
and $\chi(G-v) < k$ for each $v\in V(G)$. If the complement of $G$ is connected, we say $G$ is {\em anti-connected}.

\begin{theorem}{\rm(\cite{MS})}\label{thmMS}
Let $G$ be a $k$-vertex-critical graph such that $G$ is anti-connected. Then for any $x\in V(G)$, the graph $G-x$ has a $(k-1)$-coloring in which every color class contains at least {\rm2} vertices.
\end{theorem}

Theorem \ref{obiminor2} follows immediately from Theorem \ref{main-lemma2}.
\begin{theorem} \label{main-lemma2}
Let $G$ be an $n$-vertex graph with $\alpha(G)=2$. For any positive integer $\ell$ with $\ell < \chi(G)$, the graph $G$ contains a special odd model of $K_{\ell,\chi(G)- \ell}$. 
\end{theorem}

\begin{proof}
Assume not. Let $G$ be a counterexample to Theorem \ref{main-lemma2} with $|V(G)|$ as small as possible. Then there exists an integer $\ell$ with $\ell<\chi(G)$ such that $G$ contains no special odd model of $ K_{\ell,\chi(G)- \ell}$. Since $\Delta(G)\geq \chi(G)-1$  implies that $G$ contains a special odd model of  $K_{1,\chi(G)-1}$, we have $\ell\geq2$, where $\Delta(G)$ denotes the maximum degree of $G$.

We claim that $G$ is $\chi(G)$-vertex-critical. Assume not. Then $\chi(G-x)=\chi(G)$ for some vertex $x\in V(G)$. Moreover, by the minimality of $G$, for any positive integer $\ell<\chi(G-x)$, the graph  $G-x$ contains a special odd model of $ K_{\ell,\chi(G)- \ell}$ as $\chi(G-x)=\chi(G)$, which is a contradiction. Hence, $G$ is $\chi(G)$-vertex-critical.

We claim that $G$ is anti-connected. Suppose not. Then there is a partition $(V_1,V_2)$ of $V(G)$ such that $V_1$ is complete to $V_2$. Set $G_1:=G[V_1]$ and $G_2:=G[V_2]$. Then $\chi(G)=\chi(G_1)+\chi(G_2)$.
By the minimality choice of $G$, for any positive integers $1\leq i\leq2$ and $\ell_i<\chi(G_i)$, the graph $G_i$ contains an odd model of $K_{\ell_i,\chi(G_i)- \ell_i}$ such that all branch sets that are isolated vertices are colored by the same color, say color 1.  Without loss of generality we may further assume that $\ell_1$, $\ell_2$ are chosen with $\ell=\ell_1+\ell_2$. Since  $\ell\geq2$, such choice exists.
Hence, $G$ contains a special odd model of $K_{\ell_1+\ell_2,\chi(G_1)+\chi(G_2)- (\ell_1+\ell_2)}=K_{\ell,\chi(G)- \ell}$, contradicting that $G$ is a counterexample to Theorem \ref{main-lemma2}. Hence,  $G$ is anti-connected.

By Theorem \ref{thmMS} and the two claims proved in the last two paragraphs, we have that $|V(G)|\geq2\chi(G)-1$. Thus, $2\chi(G)-1 \leq|V(G)|\leq2\chi(G)$ as $\alpha(G)=2$.
When $|V(G)|=2\chi(G)$, for any vertex $x\in V(G)$, 
we have that $$2\chi(G)-1=|V(G-x)|\leq 2\chi(G-x)=2(\chi(G)-1),$$ which is a contradiction.
When $|V(G)|=2\chi(G)-1$, by Theorem \ref{main-lemma}, $G$ contains a special odd model of $K_{\ell,\lceil\frac{|V(G)|}{2}\rceil- \ell}=K_{\ell,\chi(G)- \ell}$, which is a contradiction.
\end{proof}


	
	
	\vspace{5pt}

\end{document}